\documentclass{article}
\usepackage[a4paper, margin=1in]{geometry}
\usepackage{authblk}
\usepackage{version}
\usepackage{lmodern}
\usepackage[utf8]{inputenc}
\usepackage[T1]{fontenc}
\usepackage[parfill]{parskip}
\usepackage{float}
\usepackage{amsfonts}
\usepackage{amsmath}
\usepackage{amssymb}
\usepackage{amsthm}
\usepackage{accents}
\usepackage{stmaryrd}
\usepackage{mathtools}
\usepackage{enumitem}
\usepackage{hyperref}
\usepackage[capitalise]{cleveref}
\usepackage{amsthm}
\usepackage{thmtools}
\hypersetup{hidelinks}
\DeclareMathOperator{\rank}{rank}

\newcommand{\calH}{\mathcal{H}}
\newcommand{\calE}{\mathcal{E}}
\newcommand{\pth}[1]{\!\left( #1 \right)}

\newcommand{\bigO}[1]{\mathcal{O}\pth{#1}}
\newcommand{\bigTheta}[1]{\Theta\pth{#1}}
\newcommand{\bigOmega}[1]{\Omega\pth{#1}}
\newcommand{\set}[1]{\left\{#1\right\}}

\newtheorem{theorem}{Theorem}
\newtheorem{proposition}[theorem]{Proposition}
\newtheorem{lemma}[theorem]{Lemma}
\newtheorem{corollary}[theorem]{Corollary}
\newtheorem{claim}{Claim}[theorem]
\newtheorem{conjecture}[theorem]{Conjecture}

\newtheorem{problem}{Problem}

\newenvironment{subproof}[1][\proofname]{%
	\begin{proof}[#1]%
	}{%
	\end{proof}%
}

\title{Frugal colourings of graphs via sparse hypergraph colouring}
\author{Quentin Chuet}
\affil{LIRMM, Univ Montpellier, CNRS, Université de Montpellier Paul-Valéry,\\Montpellier, France}
\date{}

\begin{document}
\maketitle

\begin{abstract}
	\leftskip=0.5cm
	\rightskip=0.5cm
	A proper colouring of a graph $G$ is $\beta$-frugal if every colour appears at most $\beta$ times in the neighbourhood of each vertex. Let $\chi_\beta(G)$ denote the minimum number of colours required for a $\beta$-frugal colouring of $G$. For a fixed frugality $\beta$, Hind et al. showed that $\chi_\beta(G) = \mathcal{O}(\Delta(G)^{1 + 1/\beta})$, and a bipartite construction of Alon certifies the tightness of this upper bound up to a constant factor. Using a sparse hypergraph colouring theorem of Li and Postle, we show that for all fixed $\beta \ge 2$ and $t\ge 2$, if $G$ excludes either $C_{2t}$ or $K_{\beta,t}$ as a subgraph, then $\chi_\beta(G) = \mathcal{O}(\Delta(G)^{1 + 1/\beta} / (\log\Delta(G))^{1/\beta})$. Furthermore, we show that these upper bounds are tight up to a constant factor due to the existence of graphs $G$ with arbitrarily large maximum degree $\Delta$ and girth such that $\chi_\beta(G) = \Omega(\Delta^{1 + 1/\beta} / (\log\Delta)^{1/\beta})$. Finally, we show that excluding an arbitrary bipartite subgraph does not necessarily reduce the order of growth of $\chi_\beta$, as there exist $K_{4^{\beta+1},4^{\beta+1}}$-free graphs $G$ with arbitrarily large maximum degree $\Delta$ such that $\chi_\beta(G) = \Omega(\Delta^{1 + 1/\beta})$.
\end{abstract}

\section{Introduction}

\subsection{Background}
We say that a graph is $H$-free if it does not contain $H$ as a (not necessarily induced) subgraph. Johansson \cite{johansson1996choice} showed that triangle-free graphs of maximum degree $\Delta$ have chromatic number $\bigO{\Delta / \log\Delta}$, and the leading constant was later refined by Molloy \cite{molloy2019list}.

\begin{theorem}[Molloy, 2019]\label{thm:molloy}
	Let $G$ be a triangle-free graph of maximum degree $\Delta$. Then \[\chi(G) \le (1 + o(1))\frac{\Delta}{\log \Delta}.\]
\end{theorem}

This general upper bound is tight up to a constant factor, as certified by the existence of graphs with arbitrarily large maximum degree and girth that require at least half as many colours. Such graphs were first obtained by Bollobás \cite[Corollary 4]{bollobas1981independence} by analysing random graphs, and there are no known deterministic constructions with a comparably large chromatic number in terms of the maximum degree.

\begin{theorem}[Bollobás, 1981]\label{thm:chromatic-girth}
	For every $\Delta \ge 2$ and $g \ge 3$, there exists a graph of maximum degree $\Delta$ and girth $g$ such that \[\chi(G) > \frac{\Delta}{2\log\Delta}.\]
\end{theorem}

Alon, Krivelevich, and Sudakov \cite{alon1999coloring} generalised Johansson's theorem on triangle-free graphs to \emph{locally sparse} graphs, that is, graphs where every neighbourhood spans $o(\Delta^2)$ edges.

\begin{theorem}[Alon, Krivelevich, Sudakov, 1999]\label{thm:aks}
	Let $G$ be a graph of maximum degree $\Delta$. If there exists $2 \le f \le \Delta^2$ such that every vertex $v$ is incident to at most $\Delta^2/f$ triangles, then \[\chi(G) = \bigO{\frac{\Delta}{\log f}}.\]
\end{theorem}

In particular, using bounds on the density of $K_{t,t}$-free graphs, they showed that for any fixed $t\ge 1$, every $K_{1,t,t}$-free graph of maximum degree $\Delta$ has chromatic number $\bigO{\Delta / \log\Delta}$. In general, if $H$ is a fixed graph and $G$ is $H$-free, it is known that $\chi(G) = \bigO{\Delta \log\log\Delta / \log\Delta}$ \cite{johansson1996choice,molloy2019list}. Alon, Krivelevich, and Sudakov conjectured that the $\bigO{\log\log\Delta}$ factor can be removed.

\begin{conjecture}[Alon, Krivelevich, Sudakov, 1999]\label{conj:aks}
	Let $H$ be a fixed graph. If $G$ is an $H$-free graph of maximum degree $\Delta$, then \[\chi(G) = \bigO{\frac{\Delta}{\log\Delta}}.\]
\end{conjecture}

The first advancement towards proving \cref{conj:aks} since it was posed was made recently by Dhawan, Janzer, and Methuku \cite{dhawan2025independent}, who proved the conjecture when $H = K_{t,t,t}$, with $t \ge 1$. The first open case of \cref{conj:aks} is $H = K_4$.

\subsection{Frugal colouring}
We consider an analogue of \cref{conj:aks} for a variant of proper colouring introduced by Hind, Molloy, and Reed \cite{hind1997colouring}. A proper colouring of a graph $G$ is $\beta$-frugal if every colour appears at most $\beta$ times in the neighbourhood of each vertex. The minimum number of colours needed for a $\beta$-frugal colouring of $G$ is denoted $\chi_\beta(G)$. Throughout this paper, we consider the frugality $\beta$ to be a fixed constant; for results concerning $\beta$-frugal colourings where $\beta$ depends on the maximum degree, we refer to \cite{molloy2010asymptotically}. Hind, Molloy, and Reed proved the following bounds.

\begin{theorem}[Hind, Molloy, Reed, 1997]\label{thm:hmr}
	Fix $\beta \ge 1$. Let $G$ be a graph of sufficiently large maximum degree $\Delta$. Then 
	\[\chi_\beta(G) \le \frac{e^3}{\beta}\Delta^{1 + 1/\beta}.\]
	Furthermore, for arbitrarily large $\Delta$, there exists a bipartite graph $G$ of maximum degree $\Delta$ such that $\chi_\beta(G) \ge \Delta^{1 + 1/\beta}/2\beta$.
\end{theorem}

We refer to \cite[Theorem 7]{wanless2022general} for a refinement of the upper bound. The construction for the lower bound in \cref{thm:hmr}, credited to Noga Alon, is the point-hyperplane bipartite incidence graph of the projective geometry $PG(\beta+1,q)$. With a more careful analysis, Kang and Müller \cite[Corollary 3.4]{kang2011frugal} showed that the lower bound can be improved to $(1 + o(1))\Delta^{1 + 1/\beta}/\beta$ as $\Delta \to \infty$.

A simpler construction, appearing in the work of Yuster \cite{yuster1998linear} for the case $\beta = 2$, also certifies a $\bigOmega{\Delta^{1 + 1/\beta}}$ lower bound for $\beta$-frugal colouring of graphs of maximum degree $\Delta$, although it is not bipartite and provides a weaker multiplicative constant. Consider the graph $G_n^{\beta+1}$ with vertex set $[n]^{\beta+1}$ and where two vertices are adjacent if they share at least one coordinate. One easily shows that $G_n^{\beta+1}$ has maximum degree $\Delta \sim (\beta+1)n^{\beta}$, and that every $\beta$-frugal colouring of $G_n^{\beta+1}$ uses each colour at most $\beta$ times, thus $\chi_\beta(G_n^{\beta + 1}) \ge n^{\beta+1}/\beta \ge (1 + o(1))\Delta^{1 + 1/\beta} / (\beta(\beta+1)^{1 + 1/\beta})$. We remark that $G_n^{\beta + 1}$ is isomorphic to the line graph of the complete $(\beta+1)$-uniform $(\beta+1)$-partite hypergraph with all parts of size $n$; a similar argument works for the line graphs of other dense $(\beta+1)$-uniform hypergraphs, such as the complete $(\beta+1)$-uniform hypergraph on $n$ vertices.

Note that $1$-frugal colourings correspond exactly to distance-$2$ colourings, that is, colourings in which no pair of vertices at distance at most $2$ is monochromatic; it is clear that $\chi_1(G) = \chi(G^2)$. If $G$ has maximum degree $\Delta$, one easily obtains a greedy distance-$2$ colouring using $\Delta^2 + 1$ colours, and this is asymptotically optimal as there exist bipartite graphs of girth $6$ --- namely the bipartite incidence graphs of projective planes --- that require $(1 + o(1))\Delta^2$ colours. Alon and Mohar \cite{alon2002chromatic} showed that graphs of girth at least $7$ require significantly fewer colours for distance-$2$ colouring.

\begin{theorem}[Alon, Mohar, 2002]\label{thm:am-square}
	Let $G$ be a graph of maximum degree $\Delta$ and girth $g \ge 7$. Then \[\chi(G^2) = \bigO{\frac{\Delta^2}{\log\Delta}}.\]
	Furthermore, there exists an absolute constant $c > 0$ such that, for every $\Delta$ and $g$, there exists a graph $G$ of maximum degree $\Delta$ and girth $g$ such that $\chi(G^2) \ge c\frac{\Delta^2}{\log \Delta}$.
\end{theorem}

Kang and Pirot \cite{kang2018distance} showed that replacing the girth $7$ requirement in \cref{thm:am-square} with the condition ``$G$ is $C_{2t}$-free, for some fixed $t \ge 3$'' suffices. We extend this result to frugal colourings.

\begin{restatable}{theorem}{cycles}\label{thm:frugal-cycle}
	Fix $\beta\ge 2$ and $t \ge 2$. Let $G$ be a $C_{2t}$-free graph of maximum degree $\Delta$. Then \[\chi_\beta(G) = \bigO{\frac{\Delta^{1 + 1/\beta}}{(\log\Delta)^{1/\beta}}}.\]
\end{restatable}

Note that in the statement of \cref{thm:frugal-cycle}, excluding $C_4 = K_{2,2}$ as a subgraph suffices to obtain a polylogarithmic reduction, as opposed to distance-$2$ colouring where $\bigOmega{\Delta^2}$ colours may be required even when considering graphs of girth $6$. Somewhat surprisingly, we show that for any fixed frugality $\beta \ge 2$, we can exclude $K_{\beta,t}$ as a subgraph and obtain the same conclusion as \cref{thm:frugal-cycle}.

\begin{restatable}{theorem}{maintheorem}\label{thm:frugal-Kbt}
	Fix $\beta \ge 2$ and $t \ge 2$. Let $G$ be a $K_{\beta,t}$-free graph of maximum degree $\Delta$. Then \[\chi_\beta(G) = \bigO{\frac{\Delta^{1 + 1/\beta}}{(\log\Delta)^{1/\beta}}}.\]
\end{restatable}

We also show that \cref{thm:frugal-cycle} and \cref{thm:frugal-Kbt} are tight up to a constant factor.

\begin{restatable}{theorem}{girth}\label{thm:frugal-girth}
	Fix $\beta \ge 1$. There exists a constant $c_\beta > 0$ such that, for every $\Delta$ and $g$, there exists a graph of maximum degree $\Delta$ and girth $g$ such that \[\chi_\beta(G) \ge c_\beta \frac{\Delta^{1 + 1/\beta}}{(\log\Delta)^{1/\beta}}.\]
\end{restatable}

For an integer $\Delta$ and a graph $H$, let $\chi_\beta(\Delta, H)$ be defined as the maximum value of $\chi_\beta(G)$ over all $H$-free graphs $G$ of maximum degree at most $\Delta$. The following corollaries are direct consequences of  \cref{thm:frugal-cycle}, \cref{thm:frugal-Kbt}, and \cref{thm:frugal-girth}.

\begin{corollary}\label{cor:frugal-cycle}
	Fix $\beta \ge 2$ and $t \ge 2$. Then
	\[\chi_\beta(\Delta,C_{2t}) = \bigTheta{\frac{\Delta^{1 + 1/\beta}}{(\log\Delta)^{1/\beta}}}.\]
\end{corollary}

\begin{corollary}\label{cor:frugal-Kbt}
	Fix $\beta \ge s \ge 2$ and $t \ge 2$. Then
	\[\chi_\beta(\Delta,K_{s,t}) = \bigTheta{\frac{\Delta^{1 + 1/\beta}}{(\log\Delta)^{1/\beta}}}.\]
\end{corollary}

In contrast, if $H$ is not a bipartite graph, we know from \cref{thm:hmr} that $\chi_\beta(\Delta, H) = \bigTheta{\Delta^{1 + 1/\beta}}$. Could it be the case, then, that $\chi_\beta(\Delta,H) = o(\Delta^{1 + 1/\beta})$ for every bipartite graph $H$? This is false already for $\beta = 1$ since we know that $\chi_1(\Delta,K_{2,2}) = \bigOmega{\Delta^2}$, and for larger frugalities $\beta$ we determine the existence of a phase transition when excluding sufficiently large balanced bicliques.

\begin{theorem}\label{thm:lowerbound-biclique}
	Fix $\beta \ge 2$, and let $t \ge 4^{\beta+1}$. Then
	\[\chi_\beta(\Delta, K_{t,t}) = \bigTheta{\Delta^{1 + 1/\beta}}.\]
\end{theorem}

With a little more effort, it is possible to take $t \ge 2\cdot 4^{\beta}$ instead, but we disregard this improvement in favour of a simpler proof. It is unclear for which values of $s$ and $t$ we have $\chi_\beta(\Delta, K_{s,t}) = \bigOmega{\Delta^{1 + 1/\beta}}$; could we take $s$ polynomial in $\beta$, or $s = \beta+1$, or even $s =t=\beta+1$? More generally, taking inspiration from \cref{conj:aks}, we ask which subgraph exclusions yield a reduced order of growth for $\chi_\beta$.

\begin{problem}\label{problem:frugal}
	Fix $\beta \ge 1$. For which (bipartite) graphs $H$ do we have $\chi_\beta(\Delta, H) = o(\Delta^{1 + 1/\beta})$?
\end{problem}

Even when $\beta = 1$ (distance-2 colouring), characterising the graphs $H$ such that $\chi_1(\Delta,H) = o(\Delta^2)$ may prove to be an interesting but challenging problem. 

\subsection{Hypergraph colouring}
The proofs of \cref{thm:frugal-cycle} and \cref{thm:frugal-Kbt} use the framework of proper hypergraph colouring. We first need to introduce some notation before stating the theorem we will use. The rank of a hypergraph $\calH$ is the maximum size of its edges. An $\ell$-edge is an edge with exactly $\ell$ vertices. The $\ell$-degree of a vertex $v$ is the number of $\ell$-edges that contain $v$, and the maximum $\ell$-degree of $\calH$ is denoted $\Delta_{\ell}(\calH)$. Given a subset $S \subseteq V(\calH)$, the $\ell$-codegree of $S$ is the number of $\ell$-edges that contain $S$. For $1 \le s < \ell$, the maximum $(s,\ell)$-codegree $\Delta_{s,\ell}(\calH)$ is the maximum $\ell$-codegree over all subsets $S \subseteq V(\calH)$ of size exactly $s$; in particular, $\Delta_\ell(\calH) = \Delta_{1,\ell}(\calH)$.

A proper $k$-colouring of a hypergraph $\calH$ is a vertex colouring $V(\calH) \to [k]$ such that no edge of $\calH$ is monochromatic. The chromatic number $\chi(\calH)$ is the minimum $k$ such that a proper $k$-colouring of $\calH$ exists. All hypergraphs considered have minimum edge size at least $2$, so a proper colouring always exists (e.g. by assigning a unique colour to each vertex). Using a straightforward application of the asymmetric Lovász Local Lemma, entropy compression, or even a counting argument of Rosenfeld \cite{rosenfeld2020another}, one can show (see e.g. \cite[Lemma 8]{cooper2016coloring}) that if $\calH$ has rank $r$ and maximum $\ell$-degree $\Delta_\ell$ for $\ell \in \{2,\dots,r\}$, then \begin{equation}\label{eq:hypergraph}
	\chi(\calH) = \bigO{\max_{2 \le \ell \le r}\Delta_\ell^{1/(\ell-1)}}.
\end{equation}
With additional assumptions on the local sparsity of $\calH$, reminiscent of \cref{thm:aks}, Li and Postle \cite{li2022chromatic} improve \cref{eq:hypergraph} by a polylogarithmic factor.

\begin{theorem}[{Li, Postle, \cite[Theorem 1.8]{li2022chromatic}}]\label{thm:li-postle}
	Fix $r\ge 3$, and let $\calH$ be a rank $r$ hypergraph with maximum $\ell$-degree at most $\Delta_\ell$ for each $2 \le \ell \le r$.
	
	Let $\Delta_* \coloneqq \max\limits_{2\le \ell \le r}\Delta_{\ell}^{1/(\ell - 1)}$, and suppose that there exists $f > 1$ such that
	\begin{enumerate}[label=(\alph*)]
		\item for all $2 \le s < \ell \le r$, $\calH$ has maximum $(s,\ell)$-codegree $\Delta_{s,\ell}(\calH) \le \Delta_*^{\ell - s} / f$, and
		\item every vertex $v \in V(\calH)$ is incident to at most $\Delta_*^2/f$ copies of $K_3$ (the graph triangle).
	\end{enumerate}
Then,
	\[\chi(\calH) = \bigO{\max_{2 \le \ell \le r} \pth{\frac{\Delta_\ell}{\log f}}^{\frac{1}{\ell - 1}}} = \bigO{\frac{\Delta_*}{(\log f)^{1/(r-1)}}}.\]
\end{theorem}

We remark that a similar result was first proved for $r$-uniform hypergraphs by Cooper and Mubayi \cite{cooper2016coloring}. We also remark that Li and Postle derived \cref{thm:li-postle} from a more general criterion that involves counting ``hypergraph triangles'', cf. \cite[Theorem 1.7]{li2022chromatic}.

\section{Upper bounds}
\newcommand{\coN}{N^{\cap}}

In this section, we prove \cref{thm:frugal-cycle} and \cref{thm:frugal-Kbt} using \cref{thm:li-postle}. We first illustrate how \cref{thm:li-postle} can be directly applied in the case of $K_{2,t}$-free graphs. Throughout this section, it shall be assumed that the maximum degree $\Delta$ is sufficiently large. Given a subset of vertices $S$ in a graph $G$, the \emph{common neighbourhood} of $S$ is denoted $\coN_G(S) \coloneqq \bigcap\limits_{v \in S}N_G(v)$.

\begin{proposition}\label{prop:frugal-K2t}
	Fix $\beta \ge 2$ and $t \ge 2$. Let $G$ be a $K_{2,t}$-free graph of maximum degree $\Delta$. Then \[\chi_\beta(G) = \bigO{\frac{\Delta^{1 + 1/\beta}}{(\log\Delta)^{1/\beta}}}.\]
\end{proposition}
\begin{proof}
	Let $\calH = (V(G), \calE)$ where $\calE$ contains $E(G)$ and all $(\beta+1)$-subsets of vertices $S$ such that $S \subseteq N_G(v)$ for some $v \in V(G)$. $\calH$ has rank $\beta+1$. A proper colouring of $\calH$ is also proper for $G$ since $E(G) \subseteq \calE$, and is $\beta$-frugal since there is no vertex $v$ that has $\beta+1$ neighbours in $G$ of the same colour. Therefore, $\chi_\beta(G) \le \chi(\calH)$ (in fact there is equality).
	
	Clearly, $\calH$ has maximum $2$-degree at most $\Delta_2 \coloneqq \Delta$. Given a vertex $v \in V(G)$, there are $\deg_G(v)$ neighbours $u \in N_G(v)$, from which there are ${\deg_G(u) - 1 \choose \beta}$ ways to choose the remaining $\beta$ vertices to form a $(\beta+1)$-edge $S \in \calE$ containing $v$. Thus, $\calH$ has maximum $(\beta+1)$-degree at most $\Delta_{\beta+1} \coloneqq \Delta^{\beta + 1}$. Let $\Delta_* \coloneqq \max(\Delta_2,\Delta_{\beta+1}^{1/\beta}) = \Delta^{1 + 1/\beta}$. In order to apply \cref{thm:li-postle}, we need to bound the maximum $(s,\beta+1)$-codegree for every $2 \le s \le \beta$. Consider a subset $S$ of size $s \ge 2$; we have $|\coN_G(S)| < t$, otherwise we would find a copy of $K_{2,t}$ in $G$. Therefore, the $(\beta+1)$-codegree of $S$ is at most $(t-1) \cdot {\Delta - s \choose \beta+1-s}$. Therefore $\calH$ has maximum $(s,\beta+1)$-codegree $\Delta_{s,\beta+1}(\calH) <  t\Delta^{\beta + 1 - s} \le \Delta_*^{\beta+1-s}/f$ with $f \coloneqq \Delta^{1/\beta}/t$. As for the number of triangles incident to a given vertex, it is at most ${\Delta_2 \choose 2} < \Delta_*^2/f$.
	
	All the conditions of \cref{thm:li-postle} are met, therefore we conclude that \[\chi_\beta(G) \le \chi(\calH) = \bigO{\frac{\Delta_*}{(\log f)^{1/\beta}}} = \bigO{\frac{\Delta^{1 + 1/\beta}}{(\log \Delta)^{1/\beta}}}.\qedhere\]
\end{proof}

In order to exclude a fixed even cycle $C_{2t}$ and derive the same upper bound as \cref{prop:frugal-K2t}, we will add $2$-edges in the auxiliary hypergraph $\calH$ between pairs of vertices that share ``too many'' common neighbours in $G$ (which we call \emph{special pairs}). This will give us enough control when bounding the maximum $(s,\beta+1)$-codegrees of $\calH$, and has a negligible impact on the $2$-degrees in $\calH$ (the idea being that if too many special pairs intersect in a given vertex, then we would find a copy of $C_{2t}$ in $G$).

We will use a classical theorem of Erdős and Gallai \cite{erdos1959maximal} bounding the number of edges in a graph that excludes a path on $t$ vertices as a subgraph.

\begin{theorem}[Erdős, Gallai, 1959]\label{thm:path}
	Let $G$ be an $n$-vertex graph. If $G$ is $P_t$-free, then
	\[|E(G)| \le \frac{t-2}{2}n.\] 
\end{theorem}

We now prove \cref{thm:frugal-cycle}, of which we recall the statement below.
\cycles*
\begin{proof}
	Let $\alpha \coloneqq 2t$. A pair of non-adjacent vertices $\{u,v\} \subseteq V(G)$ is called a \emph{special pair} if they share $|\coN_G(\{u,v\})| > \alpha$ common neighbours. Let $\calH = (V(G), \calE)$ where $\calE$ contains $E(G)$, all special pairs, and all $(\beta+1)$-subsets of vertices $S$ such that $S \subseteq N_G(v)$ for some $v \in V(G)$ and such that $S$ does not contain a smaller edge. $\calH$ has rank $\beta+1$ and maximum $(\beta+1)$-degree at most $\Delta_{\beta+1} \coloneqq \Delta^{\beta+1}$.
	
	A proper colouring of $\calH$ is also proper for $G$ since $E(G) \subseteq \calE$, and is $\beta$-frugal since every $(\beta+1)$-subset $S \subseteq V(G)$ with a common neighbour contains an edge in $\calE$ and thus cannot be monochromatic. Therefore, $\chi_\beta(G) \le \chi(\calH)$.
	
	We first bound the maximum $2$-degree of $\calH$. Let $u$ be a vertex, and let $\sigma(u)$ be the set of vertices $v$ such that $\{u,v\}$ is a special pair. The $2$-degree of $u$ is precisely $\deg_G(u) + |\sigma(u)|$. Observe that $N_G(u)$ and $\sigma(u)$ are disjoint since special pairs are non-adjacent in $G$. Consider the bipartite graph $H$ induced by $G$ between $N_G(u)$ and $\sigma(u)$, i.e. where two vertices $(w,v) \in N_G(u) \times \sigma(u)$ are adjacent in $H$ if and only if $wv \in E(G)$. For every $v \in \sigma(u)$, since $\{u,v\}$ forms a special pair, we have $\deg_H(v) = |\coN_G(\{u,v\})| > \alpha$. Therefore, the number of edges in $H$ is at least
	\begin{equation}\label{eq:H-lower}
		|E(H)| > \alpha \cdot |\sigma(u)|.
	\end{equation}
	$H$ does not contain a copy of $P_{2t}$ as a subgraph, otherwise we would find a copy of $P_{2t - 1}$ in $H$ with both endpoints in $N_G(u)$, which we could then close into a copy of $C_{2t}$ in $G$ using $u$. Therefore we apply \cref{thm:path} with $n \coloneqq \deg_G(u) + |\sigma(u)|$ to deduce that the number of edges in $H$ is at most
	\begin{equation}\label{eq:H-upper}
		|E(H)| < t (\Delta + |\sigma(u)|).
	\end{equation}
	
	Putting \cref{eq:H-lower} and \cref{eq:H-upper} together, we obtain the inequality
	\begin{align*}
		|\sigma(u)| < \frac{t \Delta}{\alpha - t} = \Delta.
	\end{align*}
	Therefore the maximum $2$-degree of $\calH$ is at most $\Delta_2 \coloneqq 2\Delta$.
	
	Let $\Delta_* \coloneqq \max(\Delta_2, \Delta_{\beta + 1}^{1/\beta}) = \Delta^{1 + 1/\beta}$. We now bound the maximum $(s,\beta+1)$-codegrees. Consider an integer $2\le s \le \beta$, and let $S$ be a $s$-subset of $V(G)$. If $S$ contains an edge in $\calE$, then no $(\beta+1)$-edge contains $S$ by definition and thus $S$ has $(\beta+1)$-codegree $0$. Otherwise, we have $|\coN_G(S)| \le \alpha$, in which case $S$ has $(\beta+1)$-codegree at most $\alpha \cdot {\Delta - s \choose \beta + 1 - s} \le 2t\Delta^{\beta + 1 - s}$. We have thus shown that $\calH$ has maximum $(s,\beta+1)$-codegree $\Delta_{s,\beta+1}(\calH) \le \Delta_*^{\beta + 1 - s}/f$ with $f = \Delta^{1/\beta}/2t$. As for the number of triangles incident to a given vertex, it is at most ${\Delta_2 \choose 2} < 2\Delta^2 < \Delta_*^2/f$.
	
	All the conditions of \cref{thm:li-postle} are met, therefore we conclude that \[\chi_\beta(G) \le \chi(\calH) = \bigO{\frac{\Delta_*}{(\log f)^{1/\beta}}} = \bigO{\frac{\Delta^{1 + 1/\beta}}{(\log \Delta)^{1/\beta}}}.\qedhere\]
\end{proof}

In order to exclude $K_{\beta,t}$ as a subgraph, we will generalise the idea of \emph{special pairs} to \emph{special $s$-sets}, along with a specific common neighbourhood threshold $\alpha_s$, for each $s \in \{2,\dots,\beta\}$. We will also need bounds on the maximum number of edges in an unbalanced bipartite graph that excludes a complete bipartite subgraph. The Zarankiewicz number $z(a,b,s,t)$ is the maximum number of edges in a bipartite graph $G$ of parts $|A|=a$ and $|B|=b$ such that $G$ does not contain a copy of $K_{s,t}$ where the $s$ vertices are in $A$ and the $t$ vertices are in $B$. The celebrated Kővári-Sós-Turán theorem \cite{kovari1954problem} provides a general upper bound for the Zarankiewicz numbers.

\begin{theorem}[Kővári, Sós, Turán, 1954]\label{thm:kst}
	For every $a \ge s \ge 2$ and $b \ge t \ge 2$,
	\[z(a,b,s,t) \le (t-1)^{1/s}(a-s+1)b^{1 - 1/s} + (s-1)b.\]
\end{theorem}

We are now ready to prove \cref{thm:frugal-Kbt}, of which we recall the statement below.
\maintheorem*
\begin{proof}
	Let $\varepsilon \coloneqq \frac{1}{4\beta^2}$, $f \coloneqq \Delta^{\varepsilon}$, $\alpha_1 \coloneqq \Delta$, and for $s \in \{2,\dots,\beta\}$ let $\alpha_s \coloneqq \Delta^{\frac{\beta+1-s}{\beta} - \varepsilon}$. A set $S \subseteq V(G)$ of size $s \in \{2,\dots,\beta\}$ is called a \emph{special $s$-set} if $S$ is independent in $G$, has $|\coN_G(S)| > \alpha_s$ common neighbours, and does not contain a smaller special set. Let $\calH = (V(G), \calE)$ where $\calE$ contains $E(G)$, all special sets, and all $(\beta+1)$-subsets of vertices $S$ such that $S \subseteq N_G(v)$ for some $v \in V(G)$ and such that $S$ does not contain a smaller edge. $\calH$ has rank $\beta+1$ and maximum $(\beta+1)$-degree at most $\Delta_{\beta+1} \coloneqq \Delta^{\beta+1}$. Let $\Delta_* \coloneqq \Delta_{\beta+1}^{1/\beta} = \Delta^{1 + 1/\beta}$; we will verify at the end that $\Delta_\ell(\calH)^{1/(\ell - 1)} \le \Delta_*$ for all $2 \le \ell \le \beta$, and that $\Delta_*$ is therefore defined correctly as per the statement of \cref{thm:li-postle}.
	
	A proper colouring of $\calH$ is also proper for $G$ since $E(G) \subseteq \calE$, and is $\beta$-frugal since every $(\beta+1)$-subset $S \subseteq V(G)$ with a common neighbour contains an edge in $\calE$ and thus cannot be monochromatic. Therefore, $\chi_\beta(G) \le \chi(\calH)$.
	
	We first bound the maximum $(s,\beta+1)$-codegrees. Consider an integer $2\le s \le \beta$, and let $S$ be a $s$-subset of $V(G)$. If $S$ contains an edge in $\calE$, then no $(\beta+1)$-edge contains $S$ by definition and thus $S$ has $(\beta+1)$-codegree $0$. Otherwise, we have $|\coN_G(S)| \le \alpha_s$, in which case $S$ has $(\beta+1)$-codegree at most $\alpha_s \cdot {\Delta - s \choose \beta + 1 - s} \le \Delta^{(\beta + 1 - s)(1 + 1/\beta) - \varepsilon}$. We have thus shown that $\calH$ has maximum $(s,\beta+1)$-codegree $\Delta_{s,\beta+1}(\calH) \le \Delta_*^{\beta + 1 - s}/f$.
	
	We now bound the maximum $(s,\ell)$-codegrees for $\ell \le \beta$. Consider two integers $1 \le s < \ell \le \beta$, and let $S$ be a $s$-subset of $V(G)$. If $s=1$, then $|\coN_G(S)| \le \Delta = \alpha_1$. Otherwise if $s\ge 2$, either $S$ contains an edge in $\calE$, in which case $S$ has $\ell$-codegree $0$ and we are done, or $|\coN_G(S)| \le \alpha_s$. Thus we have $|\coN_G(S)| \le \alpha_s \le \Delta^{\frac{\beta + 1 - s}{\beta}}$. Let $\sigma_\ell(S)$ be the family of all $(\ell - s)$-subsets $S' \subseteq V(G)$ such that $S \cup S'$ is a special $\ell$-set. If $(s,\ell) \ne (1,2)$, the $\ell$-codegree of $S$ is precisely $|\sigma_\ell(S)|$, and in the remaining case $(s,\ell) = (1,2)$ we must additionally count the incident edges of $E(G)$; therefore $S$ has $\ell$-codegree at most $|\sigma_\ell(S)| + \Delta$. Observe that every set $S' \in \sigma_\ell(S)$ is disjoint from $\coN_G(S)$ since $S \cup S'$ must be independent in $G$, and additionally $\coN_G(S \cup S') \subseteq \coN_G(S)$. Consider the bipartite graph $H$ with parts $A \coloneqq \coN_G(S)$ and $B \coloneqq \sigma_\ell(S)$ and where two vertices $(w,S') \in A \times B$ are adjacent iff $w \in \coN_G(S \cup S')$. Since $S \cup S'$ forms a special $\ell$-set for every $S' \in \sigma_\ell(S)$, we have $\deg_H(S') > \alpha_\ell$. Therefore, the number of edges in $H$ is at least
	\begin{equation}\label{eq:H-lower2}
		|E(H)| > \alpha_\ell \cdot |\sigma_\ell(S)|.
	\end{equation}
	Let $t' \coloneqq {t - 1 \choose \ell - s} + 1 \le 2^t$. $H$ does not contain a copy of $K_{\beta,t'}$ where the $\beta$ vertices are in $\coN_G(S)$ and the $t'$ vertices are in $\sigma_\ell(S)$, otherwise we would find a copy of $K_{\beta,t}$ in $G$ (any family of $t'$ subsets of size $\ell - s$ covers at least $t$ elements). Therefore, we apply \cref{thm:kst} with $a \coloneqq \alpha_s$ and $b \coloneqq |\sigma_\ell(S)|$ to deduce that the number of edges in $H$ is at most
	\begin{equation}\label{eq:H-upper2}
		|E(H)| < (t'-1)^{1/\beta} \alpha_s \cdot |\sigma_\ell(S)|^{1 - 1/\beta} + (\beta-1)|\sigma_\ell(S)|.
	\end{equation}

	Putting \cref{eq:H-lower2} and \cref{eq:H-upper2} together, we obtain the inequality
	\begin{align*}
		|\sigma_\ell(S)|^{1/\beta} < 2^{t/\beta} \frac{\alpha_s}{\alpha_\ell - \beta + 1} < 4^{t/\beta}\frac{\alpha_s}{\alpha_\ell}.
	\end{align*}
	Recall that $\alpha_s \le \Delta^{\frac{\beta + 1 - s}{\beta}}$ and $\alpha_\ell = \Delta^{\frac{\beta + 1 - \ell}{\beta} - \varepsilon}$ (because $\ell \ge 2$), therefore
	\begin{align*}
		|\sigma_\ell(S)| < 4^{t} \Delta^{(\beta + 1 - s) - (\beta + 1 - \ell - \beta\varepsilon)} = 4^t \Delta^{\ell - s + \frac{1}{4\beta}}.
	\end{align*}
	Therefore, $S$ has $\ell$-codegree at most $4^t\Delta^{\ell - s + \frac{1}{4\beta}} + \Delta \le \Delta_*^{\ell - s}/\Delta^{\frac{1}{2\beta}}$. Thus, we have shown that $\calH$ has maximum $(s,\ell)$-codegree $\Delta_{s,\ell}(\calH) \le \Delta_*^{\ell - s}/f$. In particular, for $\ell \in \{2,\dots,\beta\}$, we have shown that $\calH$ has maximum $\ell$-degree at most $\Delta_{\ell} \coloneqq \Delta^{\ell - 1 + \frac{1}{2\beta}} < \Delta_*^{\ell - 1}$, and therefore we indeed have \[\max_{2\le \ell \le \beta+1}\Delta_\ell^{1/(\ell-1)} = \Delta_{\beta+1}^{1/\beta} = \Delta_*.\]
	Furthermore, the number of triangles incident to a given vertex is at most ${\Delta_2 \choose 2} < \Delta^{2 + \frac{1}{\beta}} < \Delta_*^2/f$. All the conditions of \cref{thm:li-postle} are met, therefore we conclude that
	\[\chi_\beta(G) \le \chi(\calH) = \bigO{\frac{\Delta_*}{(\log f)^{1/\beta}}} = \bigO{\frac{\Delta^{1 + 1/\beta}}{(\log \Delta)^{1/\beta}}}.\qedhere\]
\end{proof}

\section{Lower bounds for high girth}

\newcommand{\bbG}{\mathbb{G}}
\newcommand{\bbF}{\mathbb{F}}
\newcommand{\pr}[1]{\mathbb{P}\left[#1\right]}
\newcommand{\esp}[1]{\mathbb{E}\left[#1\right]}

In this section, we prove \cref{thm:frugal-girth} by analysing the Erdős-Rényi random graph distribution $\bbG(n,p)$, where a graph $G \gets \bbG(n,p)$ has vertex set $[n]$, and every pair of vertices $\{u,v\}$ independently has probability $p$ of being an edge in $G$. The proofs in this section take loose inspiration from \cite{alon1991acyclic} and \cite{alon2002chromatic}.

We will use the following inequality; we include a short proof for completeness.
\begin{lemma}\label{lem:ineq}
	Let $n \ge 1$ be an integer and $0 \le x \le 1$. Then
	\[\log\pth{1 + x + \frac{x^2}{2!} + \dots + \frac{x^n}{n!}} \le x - \frac{x^{n+1}}{e(n+1)!}.\]
\end{lemma}
\begin{proof}
	The function $f(x) \coloneqq \log\pth{1 + x + \frac{x^2}{2!} + \dots + \frac{x^n}{n!}}$ defined over $[0,1]$ has derivative
	\begin{align*}
		f'(x) = \frac{1 + x + \frac{x^2}{2!} + \dots + \frac{x^{n-1}}{(n-1)!}}{1 + x + \frac{x^2}{2!} + \dots + \frac{x^n}{n!}} = 1 - \frac{\frac{x^n}{n!}}{1 + x + \frac{x^2}{2!} + \dots + \frac{x^n}{n!}} \le 1 - \frac{x^n}{en!},
	\end{align*}
	where we have used that $1 + x + \frac{x^2}{2!} + \dots + \frac{x^n}{n!} \le e^x \le e$ since $x \le 1$. We have $f(0) = 0$, therefore by integrating the previous inequality we obtain
	\[f(x) \le x - \frac{x^{n+1}}{e(n+1)!}.\qedhere\] 
\end{proof}

We first prove a lower bound on the $\beta$-frugal chromatic number of every ``large'' induced subgraph of $G \gets \bbG(n,p)$. This will allow us to subsequently remove vertices of large degree and small cycles while maintaining a large $\beta$-frugal chromatic number. For simplicity, we omit the floor function in the computations when it should technically be required; this does not affect the final result.

\begin{theorem}\label{thm:randomgraph}
	Fix $\beta \ge 1$, and fix $d$ sufficiently large. Let $n$ be a sufficiently large integer. Let $p \coloneqq \frac{d}{n}$ and sample $G \gets \bbG(n,p)$. With high probability, for every subset $U \subseteq V(G)$ of size at least $n/2$,
	\[\chi_\beta(G[U]) \ge k \coloneqq \pth{4^{\beta+5}(\beta+1)!}^{-\frac{1}{\beta}}\frac{d^{1 + 1/\beta}}{(\log d)^{1/\beta}}.\]
\end{theorem}
\begin{proof}
	Consider any fixed subset $U$ of size at least $\frac{n}{2}$, and fix any $k$-colouring $\sigma$ of $U$; there are at most $(k+1)^n \le \exp\pth{2n\log d}$ ways of choosing such a set $U$ and a $k$-colouring, by considering vertices in $V(G) \setminus U$ to receive the colour $k+1$ before being discarded. We claim that $\sigma$ is $\beta$-frugal in $G[U]$ with probability at most $\exp\pth{-4n \log d}$; thus, by a union bound over all choices of $U$ and $\sigma$, the probability that there exists a subset $U\subseteq V(G)$ of size $|U| \ge \frac{n}{2}$ such that $\chi_\beta(G[U]) < k$ is at most $\exp(-2 n \log d)$.
	
	Let $t \coloneqq \frac{n}{4k}$. We can easily check that $\frac{1}{d} < tp < 1$ using the assumption that $d$ is sufficiently large. Let $U_1 \sqcup \dots \sqcup U_k = U$ be the colour classes of $\sigma$. For each $i \in [k]$ we further subdivide $U_i$ into subsets of size exactly $t$ and discard the remaining vertices. At most $t-1$ vertices are discarded for each colour class, thus we have at least $|U| - k(t-1) \ge \frac{n}{4}$ vertices remaining, each belonging to a monochromatic subset of size exactly $t$. There are at least $\frac{n}{4t} \ge k$ such subsets; let us only consider the first $k$ of them, which we call $V_1, \dots, V_k$ (their precise colours will no longer matter).
	
	For $1 \le i < j \le k$, let $E_{i,j}$ be the event ``$G[V_i,V_j]$ has maximum degree at most $\beta$''. These events are mutually independent since they are determined by disjoint sets of edges, and all of these events must occur in order for $\sigma$ to be $\beta$-frugal.
	
	\begin{claim}\label{claim:proba}
		For all $1 \le i < j \le k$, \[\pr{E_{i,j}} \le \exp\pth{-\frac{ t^{\beta+2}p^{\beta+1}}{4(\beta+1)!}}.\]
	\end{claim}
	\begin{subproof}
		Let $H \coloneqq G[V_i, V_j]$ be the bipartite graph induced between $V_i$ and $V_j$. For $v \in V_i$, let $D_v$ be the event ``$\deg_H(v) \le \beta$''. The events $D_v$ are independent, and $E_{i,j} \subseteq \bigcap\limits_{v \in V_i} D_v$. For $v \in V_i$, we have
		\[\pr{D_v} = \sum_{i = 0}^{\beta} {t \choose i}(1-p)^{t - i}p^i \le (1 - p)^{t - \beta}\pth{1 + tp + \frac{(tp)^2}{2!} + \dots + \frac{(tp)^{\beta}}{\beta!}}.\]
		Recall that $\frac{1}{d} < tp < 1$. By \cref{lem:ineq} and the inequality $1 - x \le e^{-x}$, we have
		\begin{align*}
			\pr{D_v} &\le \exp\pth{-p(t - \beta) + tp - \frac{(tp)^{\beta+1}}{e(\beta+1)!}}\\
				&\le \exp\pth{\beta p - \frac{(tp)^{\beta+1}}{e(\beta+1)!}}\\
				&\le \exp\pth{-\frac{(tp)^{\beta+1}}{4(\beta+1)!}} \qquad\text{for sufficiently large }n.
		\end{align*}
		We end the proof of \cref{claim:proba} by multiplying the probabilities of the events $D_{v}$.
	\end{subproof}
	We now conclude the proof of \cref{thm:randomgraph}. The probability that $\sigma$ is $\beta$-frugal in $G[U]$ is
	\begin{align*}
		\pr{\sigma \textrm{ is $\beta$-frugal in }G[U]} &\le \prod_{1\le i < j \le k}\pr{E_{i,j}}\\
		&\le \exp\pth{-\frac{t^{\beta+2}p^{\beta+1}}{4(\beta+1)!}}^{k \choose 2}\\
		&\le \exp\pth{-\frac{1}{4(\beta+1)!}\pth{\frac{n}{4k}}^{\beta+2}p^{\beta+1}\cdot \frac{k^2}{4}}\\
		&\le \exp\pth{-\frac{1}{4^{\beta+4}(\beta+1)!}n^{\beta + 2} p^{\beta + 1} k^{-\beta}}\\
		&\le \exp\pth{- 4 n \log d}.
	\end{align*}
	Thus, the desired conclusion follows by a union bound over all possible choices of $U$ and $\sigma$.
\end{proof}

We can now prove \cref{thm:frugal-girth}, of which we recall the statement below.

\girth*

\newcommand{\bD}{\mathbf{D}}
\newcommand{\bC}{\mathbf{C}}
\begin{proof}
	Fix $d$ and $g$, and let $n$ be sufficiently large. Let $p \coloneqq \frac{d}{n}$ and sample $G \gets \bbG(n,p)$.
	
	Let $\bD$ be the random variable that counts the number of vertices of degree at least $10d$ in $G$. The probability that a vertex $v \in V(G)$ has degree at least $10d$ is at most ${n \choose 10d}p^{10d} \le e^{-10d}$, using the inequality ${n \choose k} \le \pth{\frac{en}{k}}^k$. Therefore $\esp{\bD} \le \frac{n}{100}$.
	
	For $\ell \in \{3,\dots,g\}$, let $\bC_\ell$ be the random variable that counts the number of $\ell$-cycles in $G$, and let $\bC \coloneqq \bC_3 + \dots + \bC_g$. For $\ell \in \{3,\dots,g\}$, we have $\esp{\bC_\ell} < n^\ell \cdot p^\ell = d^\ell$, thus $\esp{\bC} < d^3 + d^4 + \dots + d^g$. For $n \gg d^{g}$, we therefore have $\esp{\bC} \le \frac{n}{100}$.
	
	Let $c \coloneqq (4^{\beta + 5}(\beta+1)!)^{-\frac{1}{\beta}}$. We have
	\begin{itemize}
		\item $\pr{\bD \ge \frac{n}{10}} \le \frac{1}{10}$ and $\pr{\bC \ge \frac{n}{10}} \le \frac{1}{10}$, using Markov's inequality, and
		\item by \cref{thm:randomgraph}, w.h.p., every subset $U \subseteq V(G)$ of size at least $\frac{n}{2}$ satisfies $\chi_\beta(G[U]) \ge c \frac{d^{1 + 1/\beta}}{(\log d)^{1/\beta}}$.
	\end{itemize}

	Thus, both conditions are fulfilled simultaneously with nonzero probability; we fix such a realisation $G$, remove all vertices of degree at least $10d$ and remove one vertex per cycle of length less than $g$. We obtain a subset $U \subseteq V(G)$ of size at least $n - 2\frac{n}{10} \ge \frac{n}{2}$ such that
	\begin{itemize}
		\item $G[U]$ has maximum degree $\Delta \le 10d$, girth at least $g$, and
		\item $\chi_\beta(G[U]) \ge c \frac{d^{1 + 1/\beta}}{(\log d)^{1/\beta}}$.
	\end{itemize}
	Therefore, the conclusion of \cref{thm:frugal-girth} holds by taking $c_\beta \coloneqq \frac{1}{100}(4^{\beta + 5}(\beta+1)!)^{-\frac{1}{\beta}}$.
\end{proof}

\section{Lower bounds for large biclique exclusions}
\newcommand{\bfA}{\mathbf{A}}
In this section, we show that excluding a sufficiently large balanced biclique has no effect on the order of growth of $\chi_\beta$. Given a prime power $q$, let $\bbF_q^{n \times m}$ denote the set of $n \times m$ matrices over $\bbF_q$. A random $n \times m$ matrix $\bfA$ over $\bbF_q$ is \emph{uniformly random} if it is uniformly distributed in $\bbF_q^{n \times m}$. We shall use the following standard fact (see e.g. \cite[Lemma C.1]{forbes2014dimension}).

\begin{lemma}\label{lem:randommatrix}
	Let $\bfA$ be a uniformly random $n \times m$ matrix over $\bbF_q$.
	\begin{itemize}
		\item If $X \in \bbF_q^{k \times n}$ has linearly independent rows, then $X\bfA$ is uniformly distributed in $\bbF_q^{k \times m}$.
		\item Dually, if $Y \in \bbF_q^{m \times k}$ has linearly independent columns, then $\bfA Y$ is uniformly distributed in $\bbF_q^{n \times k}$.
	\end{itemize}  
\end{lemma}

\cref{thm:lowerbound-biclique} follows immediately from the following theorem.

\begin{theorem}\label{thm:lowerbound-biclique2}
	Fix $\beta \ge 2$ and $t \coloneqq 4^{\beta+1}$. For every $q = 2^r$ with $r \ge 3$, there exists a bipartite $K_{t,t}$-free graph $G$ with parts $(P_L,P_R)$ of size $2q^{\beta+1}-1$ such that $\Delta(G) < 4q^\beta$ and every subset of $\beta+1$ vertices in $P_L$ shares a common neighbour in $P_R$; in particular, we have $\chi_\beta(G) \ge \frac{2q^{\beta+1}-1}{\beta} = \bigOmega{\frac{\Delta(G)^{1 + 1/\beta}}{\beta}}$.
\end{theorem}

The construction is a random binary analogue of the bipartite incidence graph of the projective geometry $PG(\beta+1,q)$; we deterministically preserve the dimension argument guaranteeing common neighbours for any $(\beta+1)$-subset of vertices in one part, but replace the dot product over $\bbF_{q}$ with random binary bilinear forms to control the degrees and bicliques with high probability.

\begin{proof}
	Let $d \coloneqq r(\beta+1) + 1$, and $P \coloneqq \bbF_2^d \setminus \{0\}$. Let $P_L \coloneqq \{x_L : x \in P\}$ and $P_R \coloneqq \{x_R : x \in P\}$; these will be the two parts of the bipartite graph $G$, and they individually have size $2^{d}-1 = 2q^{\beta+1}-1$. Throughout the proof, all vector spaces and dimensions are taken over $\bbF_2$.
	
	Let $\bfA_1,\dots,\bfA_r$ be independent uniformly random $d \times d$ matrices over $\bbF_2$, and define the bilinear map $B \colon F_2^d \times F_2^d \to F_2^r$ by
	\[
		B(x,y) \coloneqq \begin{pmatrix}
			x^\top \bfA_1y\\
			\vdots\\
			x^\top \bfA_ry
		\end{pmatrix}.
	\]
	Each pair $(x_L,y_R) \in P_L \times P_R$ forms an edge in $G$ if and only if $B(x,y) = 0$. We first show that, deterministically, every  $(\beta+1)$-subset of vertices in $P_L$ shares a common neighbour in $P_R$. Indeed, for any fixed vectors $x^1,\dots,x^{\beta+1} \in P$, the linear map $f : \bbF_2^d \to (\bbF_2^r)^{\beta+1}$ defined by
	\[
		f(y) \coloneqq (
			B(x^1,y),\dots, B(x^{\beta+1},y))
	\]
	has domain dimension $d$ and codomain dimension $(\beta+1)r = d-1$, and therefore the kernel of $f$ contains a nonzero vector $y^* \in P$; it follows that $B(x^1,y^*) = \dots = B(x^{\beta+1},y^*) = 0$ and therefore $y^*_R$ is a common neighbour of $x^1_L,\dots,x^{\beta+1}_L$.
	
	We now show that with positive probability, the desired bounds on the degrees and bicliques hold.
	
	\begin{claim}\label{claim:degree-union-bound}
		$\pr{\Delta(G) \ge 4q^\beta} \le q^{1-\beta} < \frac{1}{4}$.
	\end{claim}

	\begin{subproof}
		For $x \in P$ (fixed), consider the partial application
		$B(x,\cdot) \coloneqq \begin{pmatrix}
			x^\top \bfA_1\\
			\vdots\\
			x^\top \bfA_r
		\end{pmatrix}$.
	
		$B(x,\cdot)$ is a uniformly random binary $r \times d$ matrix over $\bbF_2$. If $B(x,\cdot)$ has rank at most $r-2$, then there exist two linearly independent binary vectors $u,v \in \bbF_2^r$ such that $u^\top B(x,\cdot) = v^\top B(x,\cdot) = 0$.
		
		For every fixed pair $u,v$ of linearly independent vectors in $\bbF_2^r$, of which there are fewer than $2^{2r}$, we have $
		\pr{\begin{pmatrix}u^\top\\v^\top\end{pmatrix} B(x,\cdot) = 0} = 2^{-2d}
		$
		by \cref{lem:randommatrix}. By a union bound over all choices of $x,u,v$, and using the assumption $\beta \ge 2$, we conclude that
		\[
			\pr{\text{there exists $x \in P$ such that} \rank B(x,\cdot) < r-1} < 2^{d + 2r - 2d} = 2^{-(\beta-1)r-1}.
		\]
		Similarly, letting $B(\cdot, y) \coloneqq \begin{pmatrix}
			\bfA_1y &\cdots &\bfA_ry
		\end{pmatrix}$ for $y \in P$, we have
		\[
			\pr{\text{there exists $y \in P$ such that} \rank B(\cdot,y)< r-1} < 2^{d + 2r - 2d} = 2^{-(\beta-1)r - 1},
		\]
		therefore, with probability at least $1 - 2\cdot2^{-(\beta-1)r-1} = 1 - q^{1-\beta}$, every matrix $B(x,\cdot)$ and $B(\cdot, y)$ has rank at least $r-1$, and in that case we obtain the desired bound on the maximum degree:
		\begin{itemize}
			\item for $x_L \in P_L$, the neighbours of $x_L$ are precisely the nonzero vectors in the right kernel of $B(x,\cdot)$, which has dimension $d-\rank B(x,\cdot)$. Therefore, $\deg(x_L) < 2^{d-r+1} = 4q^\beta$;
			\item for $y_R \in P_R$, the neighbours of $y_R$ are precisely the nonzero vectors in the left kernel of $B(\cdot, y)$, which has dimension $d-\rank B(\cdot, y)$. Therefore, $\deg(y_R) < 2^{d-r+1} = 4q^\beta$.\qedhere
		\end{itemize}
	\end{subproof}
	
	\begin{claim}\label{claim:Ktt-union-bound}
		$\pr{K_{t,t} \subseteq G} \le 2^{(2-r)(2\beta+3)} < \frac{1}{4}$.
	\end{claim}
	
	\begin{subproof}
		Let $k \coloneqq 2\beta +3$, and recall that $t \coloneqq 4^{\beta+1} = 2^{k-1}$.
		
		Suppose $X_L = \set{x_L^1,\dots,x_L^{t}}$ and $Y_R = \set{y_R^1,\dots,y_R^{t}}$ form a biclique in $G$; we have $B(x^i,y^j) = 0$ for all $i,j \in [t]$. $X$ and $Y$ both span subspaces of dimension at least $k$, since a $(k-1)$-dimensional subspace of $\bbF_2^d$ only has $2^{k-1}-1$ nonzero vectors; thus, we can find two subsets $\set{u^{1}, \dots, u^{k}} \subset X$ and $\set{v^{1}, \dots, v^{k}} \subset Y$ of linearly independent vectors, which must therefore satisfy, for all $i \in [r]$,
		\[\begin{pmatrix}
			{u^1}^\top\\
			\vdots\\
			{u^k}^\top
		\end{pmatrix} \bfA_i \begin{pmatrix}
		v^1 &\cdots &v^k
	\end{pmatrix} = 0.\]
		For a fixed choice $U = \set{u^{1}, \dots, u^{k}}$ and $V = \set{v^{1}, \dots, v^{k}}$ of linearly independent vectors, the previous equality holds with probability $2^{-k^2}$ since the resulting matrix is uniformly distributed in $\bbF_2^{k \times k}$ by two applications of \cref{lem:randommatrix}, and the same holds independently for every $i \in [r]$ as $\bfA_1,\dots,\bfA_r$ are themselves independent. By a union bound over all choices for $(U,V)$, of which there are fewer than $2^{2dk}$, we conclude that $\pr{K_{t, t} \subseteq G} \le 2^{2dk - rk^2} \le 2^{(2-r)(2\beta+3)}$.
	\end{subproof}

	By \cref{claim:degree-union-bound} and \cref{claim:Ktt-union-bound}, we conclude that with probability at least $\frac{1}{2}$, the graph $G$ has maximum degree $\Delta(G) < 4q^\beta$ and is $K_{t,t}$-free, which ends the proof.
\end{proof}

We observe that the construction used in the previous proof deterministically contains $K_{2^\beta-1,2q-1}$ as a subgraph, as certified by considering any $\beta$-dimensional subspace of $\bbF_2^d$.

\section{Concluding remarks}

\newcommand{\calF}{\mathcal{F}}

The colourings considered in this work belong to a wider family of colourings with bichromatic constraints. Given a (fixed) family of connected bipartite graphs $\calF$, a $(2,\calF)$-avoiding colouring of a graph $G$ is a proper colouring such that $G$ does not contain a $2$-coloured copy of $H \in \calF$ as a subgraph. When $\calF$ consists of a single graph $H$, we will refer to $(2,H)$-avoiding colourings. $\beta$-frugal colourings are precisely $(2,K_{1,\beta + 1})$-avoiding colourings. $(2,P_4)$-avoiding colourings are known as star colourings (see e.g. \cite{fertin2004star}) and $(2,\{C_{2\ell} : \ell \ge 2\})$-avoiding colourings are known as acyclic colourings (see e.g. \cite{alon1991acyclic}). Let $\chi_{2,\calF}(G)$ denote the minimum number of colours needed for a $(2,\calF)$-avoiding colouring of $G$, and let $\chi_{2,\calF}(\Delta)$ be the maximum $\chi_{2,\calF}(G)$ over all graphs $G$ of maximum degree $\Delta$. Aravind and Subramanian \cite{aravind2011bounds,aravind2013forbidden} obtained almost tight bounds on $\chi_{2,\calF}(\Delta)$ for all families $\calF$ of connected bipartite graphs.

\begin{theorem}[Aravind, Subramanian, 2011 \& 2013]\label{thm:F-avoiding}
	Let $\calF$ be a fixed family of connected bipartite graphs, and let $m \ge 2$ be the smallest number of edges of a graph $H \in \calF$. Then
	\[\bigOmega{\frac{\Delta^{1 + \frac{1}{m-1}}}{(\log \Delta)^{\frac{1}{m-1}}}} \le \chi_{2,\calF}(\Delta) \le \bigO{\Delta^{1 + \frac{1}{m-1}}}.\]
\end{theorem}

By \cref{thm:hmr}, we know that the upper bound in \cref{thm:F-avoiding} is tight when $\calF = \{K_{1,\beta+1}\}$, i.e. $\chi_{2,K_{1,\beta+1}}(\Delta) = \bigTheta{\Delta^{1 + 1/\beta}}$. For all bipartite connected graphs $H$ other than stars, the precise order of magnitude of $\chi_{2,H}(\Delta)$ is unknown; the lower bound in \cref{thm:F-avoiding} is obtained by analysing random graphs. Does there exist a deterministic construction for $(2,H)$-avoiding colourings that would certify the tightness of the upper bound in \cref{thm:F-avoiding}?

\begin{problem}\label{prob:B}
	Let $H$ be a fixed connected bipartite graph with $m \ge 2$ edges. Is it true that \[\chi_{2,H}(\Delta) = \bigTheta{\Delta^{1 + \frac{1}{m-1}}}?\]
\end{problem}

The second problem we ask on this topic, related to the present work, is the possibility of reducing the number of colours when considering sufficiently sparse graphs. We propose girth as the simplest measure of sparsity to consider, but excluding a fixed subgraph would also be of interest. 

\begin{problem}\label{prob:C}
	Let $H$ be a fixed tree with $m \ge 2$ edges, and $g$ sufficiently large. Let $G$ be a graph of maximum degree $\Delta$ and girth at least $g$. Is it true that
	\[\chi_{2,H}(G) = o\pth{\Delta^{1 + \frac{1}{m-1}}}?\]
\end{problem}

We believe that the lower bound in \cref{thm:F-avoiding} is the correct estimate for sufficiently sparse graphs. Indeed, the proof of \cref{thm:randomgraph} can be adapted to $(2,H)$-avoiding colourings, and for every fixed tree $H$ with $m \ge 2$ edges, it is possible to show that there exist graphs $G$ of arbitrarily large maximum degree $\Delta$ and girth such that
\[\chi_{2,H}(G) = \bigOmega{\frac{\Delta^{1 + \frac{1}{m-1}}}{(\log \Delta)^{\frac{1}{m-1}}}}.\]

For both \cref{prob:B} and \cref{prob:C}, a positive answer for star colourings specifically ($H = P_4$) would already be an interesting result.

\section*{Additional remarks and subsequent developments}
Shortly after this work appeared online, we were kindly informed by Ross Kang that an independent work of a similar nature had been done in the Master's thesis of Floor Beks \cite{beks2025colouring}, where it was in particular proved that the upper bound in \cref{thm:frugal-cycle} holds in the special case $\beta=2$ with the stronger requirement that $G$ has girth $7$.

We note that since the first version of this work appeared, Lu, Ying, and Song \cite{lu20262} have determined a more general sufficient condition for \cref{problem:frugal}, and answered \cref{prob:C} affirmatively. The preceding concluding remarks are retained from the first version.

\section*{Statement on AI use}

The first version of this work, which contained all results except \cref{thm:lowerbound-biclique}, was prepared without AI assistance. GPT-6 Astra was subsequently used to proofread the manuscript and to find a construction and proof establishing \cref{thm:lowerbound-biclique}; the author has verified, simplified, and rewritten the proof. The proposed construction does not seem to be particularly novel; the methodology is similar to previous applications of random multilinear maps to extremal combinatorics, such as \cite{forbes2014dimension} and \cite{ConlonPohoataZakharov2021}. We note that GPT-6 Astra reported an alternative $K_{t,t}$-free construction with large $\beta$-frugal chromatic number where $t = 2^{\mathcal{O}(\beta^2)}$, based on subspace evasive sets (see  \cite{dvir2012subspace} and \cite[Corollary 3.4]{guruswami2013explicit}).

\section*{Acknowledgments} 
The first version of this work was completed during the author's doctoral studies at LISN, Université Paris-Saclay. The author thanks François Pirot for helpful discussions and suggestions.

\bibliographystyle{alpha}
\bibliography{references}

\end{document}